\documentclass[11pt,english]{article}
\usepackage{amsfonts,amsmath,amsthm,amscd,amssymb,latexsym}

    \usepackage[T2A]{fontenc}
    \usepackage[cp1251]{inputenc}
    \usepackage[russian]{babel}
    \usepackage{graphicx}

\begin{document}

\title{О ГЁЛЬДЕРОВОСТИ ОТОБРАЖЕНИЙ В ОБЛАСТЯХ И НА ГРАНИЦАХ}

\author{Владимир И.\,Рязанов, Руслан Р.\,Салимов, \\ Евгений А.\,Севостьянов}

\theoremstyle{plain}
\newtheorem{theorem}{Теорема}[section]
\newtheorem{lemma}{Лемма}[section]
\newtheorem{proposition}{Предложение}[section]
\newtheorem{corollary}{Следствие}[section]
\theoremstyle{definition}

\newtheorem{example}{Пример}[section]
\newtheorem{remark}{Замечание}[section]
\newcommand{\keywords}{\textbf{Key words.  }\medskip}
\newcommand{\subjclass}{\textbf{MSC 2000. }\medskip}
\renewcommand{\abstract}{\textbf{Аннотация.  }\medskip}
\numberwithin{equation}{section}

\setcounter{section}{0}
\renewcommand{\thesection}{\arabic{section}}
\newcounter{unDef}[section]
\def\theunDef{\thesection.\arabic{unDef}}
\newenvironment{definition}{\refstepcounter{unDef}\trivlist \item[\hskip \labelsep{\bf Определение \theunDef.}]}%
{\endtrivlist}

\maketitle

\begin{abstract}
Изучаются отображения с ветвлением области евклидового пространства.
Установлены гёльдеровость и липшицевость для одного класса
пространственных отображений, характеристика которых удовлетворяет
условию типа Дини в заданной области. Кроме того, в статье найдены
условия на комплексный коэффициент уравнений Бельтрами, имеющих
вырождение равномерной эллиптичности в единичном круге, при которых
обобщенные гомеоморфные решения этого уравнения непрерывны по
Гёльдеру в точках границы.
\end{abstract}

\section{Введение} В работе \cite{GG} получены неравенства типа Липшица для
квазиконформных отображений единичного шара в себя, в связи с чем
отметим также публикацию~\cite{Sal}. В настоящей заметке мы
распространим эти результаты на отображения, допускающие наличие
точек ветвления, при этом, характеристике квазиконформности
отображений разрешается быть неограниченной. Ниже нами использовано
достаточно общее условие, согласно которому указанная характеристика
обязана удовлетворять некоторому (довольно сильному) ограничению.
Как будет показано ниже, это ограничение является более общим, чем
\cite[условие~(1.9)]{GG} для отображений, чья интегральная
характеристика не меньше 1.

\medskip
Все отображения $f:D\rightarrow {\Bbb R}^n,$ участвующие в статье,
рассматриваются в области $D\subset {\Bbb R}^n,$ $n\geqslant 2,$ и
предполагаются непрерывными. Обозначим, как обычно, $B(x_0,
r)=\left\{x\in{\Bbb R}^n: |x-x_0|< r\right\},$ ${\Bbb B}^n := B(0,
1),$ $S(x_0,r) = \{ x\,\in\,{\Bbb R}^n : |x-x_0|=r\},$ ${\Bbb
S}^{n-1}:=S(0, 1),$ $ A(x_0, r_1, r_2)=\{ x\,\in\,{\Bbb R}^n :
r_1<|x-x_0|<r_2\}.$ Пусть $E,$ $F\subset \overline{{\Bbb R}^n}$ --
произвольные множества. В дальнейшем через $\Gamma(E,F,D)$ мы
обозначаем семейство всех кривых
$\gamma:[a,b]\rightarrow\overline{{\Bbb R}^n},$ которые соединяют
$E$ и $F$ в $D,$ т.е. $\gamma(a)\in E,\,\gamma(b)\in F$ и
$\gamma(t)\in D$ при $t\in(a,\,b).$ Пусть $M$ -- конформный модуль
семейств кривых в ${\Bbb R}^n$ (см. \cite[разд.~6]{Va}),
$r_0\,=\,{\rm dist}\, (x_0\,,\partial D),$
$Q:D\rightarrow\,[0\,,\infty]$ -- измеримая по Лебегу функция,
$S_{\,i}\,=\,S(x_0,r_i),$ $i=1,2.$
Согласно~\cite[раздел~11.3]{MRSY}, отображение $f:D\rightarrow
\overline{{\Bbb R}^n}$ называется {\it кольцевым
$Q$-отоб\-ра\-же\-нием в точке $x_0\,\in\,D,$} если соотношение
\begin{equation}\label{eq1}
 M\left(f\left(\Gamma\left(S_1,\,S_2,\,A\right)\right)\right)\leqslant
\int\limits_{A} Q(x)\cdot \eta^n(|x-x_0|)\ dm(x)
\end{equation}
выполнено для любого кольца $A=A(x_0, r_1, r_2),$\, $0<r_1<r_2< r_0$
и для каждой измеримой функции $\eta : (r_1,r_2)\rightarrow
[0,\infty ]\,$ такой, что
%
%
%
$$\int\limits_{r_1}^{r_2}\eta(r)\,dr\geqslant 1\,.$$
%
Заметим, что все конформные и квазиконформные отображения
удовлетворяют соотношению~(\ref{eq1}) с некоторой функцией $Q,$
которая является, при этом, ограниченной (см., напр.,
\cite[теоремы~8.1, 8.6]{MRSY}). Напомним, что отображение
$f:D\rightarrow {\Bbb R}^n$ называется {\it открытым}, если образ
любого открытого множества $U\subset D$ является открытым множеством
в ${\Bbb R}^n,$ и {\it дискретным}, если прообраз
$f^{\,-1}\left(y\right)$ каждой точки $y\in{\Bbb R}^n$ состоит
только из изолированных точек. Имеет место следующая

\medskip
\begin{theorem}\label{th4}{\sl\,
Пусть $\alpha\in (0, 1],$ $r_0>0,$ $f:D\rightarrow B(0, r_0)$ --
открытое дискретное отображение, удовлетворяющее условию~(\ref{eq1})
в точке $x_0\in D.$ Предположим, что при некотором
$0<\varepsilon_0<{\rm dist}\,(x_0,
\partial D)$
\begin{equation}\label{eq11A}
\limsup\limits_{t\rightarrow
0}\int\limits_t^{\varepsilon_0}\left(\alpha-\frac{1}{q^{1/(n-1)}_{x_0}(r)}\right)\cdot
\frac{dr}{r}<+\infty\,.
\end{equation}
Тогда найдётся $C>0,$ зависящее только от $n,$ $r_0,$
$\varepsilon_0$ и функции $Q(x)$ такое, что в некоторой окрестности
$U_0\subset D$ точки $x_0,$ зависящей только от этой точки и функции
$Q,$
%
$$|f(x)-f(x_0)|\leqslant C|x-x_0|^{\alpha}\quad \forall\,\,x\in U_0\,.$$
 }
\end{theorem}
\section{Доказательство основного результата} Пусть $h$ -- хордальная
метрика в $\overline{{\Bbb R}^n},$
%
$$
h(x,\infty)=\frac{1}{\sqrt{1+{|x|}^2}}, \ \
h(x,y)=\frac{|x-y|}{\sqrt{1+{|x|}^2} \sqrt{1+{|y|}^2}}\,, \quad x\ne
\infty\ne y\,,
$$
%
$h(C):=\sup\limits_{x,y\in C}\,h(x,y)$ -- хордальный диаметр
множества $C\subset \overline{{\Bbb R}^n}.$ Зафиксируем
$\varepsilon_0$ из условия теоремы \ref{th4}. Заметим, прежде всего,
что для каждой точки $x\in B(x_0, \varepsilon_0)$ имеет место
неравенство
\begin{equation}\label{eq2.4.3A}
 h(f(x),f(x_0))\leqslant \frac{\alpha_n}{\delta}\,
\exp\left\{-\int\limits_{|x-x_0|}^{\varepsilon_0}
\frac{dr}{rq_{x_0}^{\frac{1}{n-1}}(r)}\right\}\,, \end{equation}
где $\delta:=h\left(\overline{{\Bbb R}^n}\setminus B(0, r)\right),$
постоянная $\alpha_n$ зависит только от $n,$ а среднее значение
$q_{x_0}(r)$ функции $Q(x)$ по сфере $S(x_0, r)$ как обычно,
определено соотношением
%
$q_{x_0}(r)=
\frac{1}{\omega_{n-1}r^{n-1}}\int\limits_{S(x_0,\,r)}Q(x)d\mathcal{H}^{n-1},$
%
где $\mathcal{H}^{n-1}$ -- $(n-1)$-мерная мера Хаусдорфа (см.
\cite[теорема~3.5.1]{Sev$_7$}). Поскольку $h(x, y)\geqslant
\frac{|x-y|}{1+r^2_0}$ при всех $x,y\in B(0, r_0),$ из
(\ref{eq2.4.3A}) получаем, что
\begin{equation}\label{eq13}
|f(x)-f(x_0)|\leqslant C_n
\exp\left\{-\int\limits_{|x-x_0|}^{\varepsilon_0}
\frac{dr}{rq_{x_0}^{\frac{1}{n-1}}(r)}\right\}\,,
\end{equation}
где постоянная $C_n$ может зависеть только от $n,$ $r_0$ и функции
$Q.$ Разделим левую и правую часть (\ref{eq13}) на
$|x-x_0|^{\alpha},$ $0<\alpha\leqslant 1.$ Тогда из (\ref{eq13})
будем иметь:
$$\frac{|f(x)-f(x_0)|}{|x-x_0|^{\alpha}}\leqslant C_n\cdot
\frac{\exp\left\{-\int\limits_{|x-x_0|}^{\varepsilon_0}
\frac{dr}{rq_{x_0}^{\frac{1}{n-1}}(r)}\right\}}{{|x-x_0|^{\alpha}}}=$$
\begin{equation}\label{eq14}
=C_n\cdot \frac{\exp\left\{-\int\limits_{|x-x_0|}^{\varepsilon_0}
\frac{dr}{rq_{x_0}^{\frac{1}{n-1}}(r)}\right\}}{{\exp\left\{-\alpha\int\limits_{|x-x_0|}^1\frac{dr}{r}\right\}}}
=\widetilde{C_n}\cdot
\exp\left\{\int\limits_{|x-x_0|}^{\varepsilon_0} \frac{\alpha
dr}{r}-\int\limits_{|x-x_0|}^{\varepsilon_0}\frac{dr}{rq_{x_0}^{\frac{1}{n-1}}(r)}\right\}=
\end{equation}
$$=\widetilde{C_n}\cdot
\exp\left\{\int\limits_{|x-x_0|}^{\varepsilon_0}\left(\alpha-\frac{1}{q^{1/(n-1)}_{x_0}(r)}\right)\cdot
\frac{dr}{r}\right\}\,,$$
где $\widetilde{C_n}=C_n/\varepsilon_0^{\alpha}.$ Однако, ввиду
условия (\ref{eq11A}) найдётся такое $M_0>0,$ зависящее только от
$n,$ $\alpha$ и функции $Q$ такое, что
$$\exp\left\{\int\limits_{|x-x_0|}^{\varepsilon_0}\left(\alpha-\frac{1}{q^{1/(n-1)}_{x_0}(r)}\right)\cdot
\frac{dr}{r}\right\}\leqslant M_0\qquad \forall\,\,x\in B(x_0,
\varepsilon_0)\setminus\{x_0\}\,.$$
Тогда из~(\ref{eq14}) вытекает соотношение $|f(x)-f(x_0)|\leqslant
C\cdot|x-x_0|^{\alpha},$
где $C=\widetilde{C_n}\cdot M_0,$ что и требовалось
установить.~$\Box$

\medskip
\begin{corollary}\label{cor1}
{\sl Если в условиях теоремы \ref{th4} вместо соотношения
(\ref{eq11A}) потребовать, что
\begin{equation}\label{eq15}
\limsup\limits_{t\rightarrow
0}\int\limits_t^{\varepsilon_0}\left(1-\frac{1}{q^{1/(n-1)}_{x_0}(r)}\right)\cdot
\frac{dr}{r}<+\infty\,,
\end{equation}
то отображение $f$ удовлетворяет оценке
%
$$|f(x)-f(x_0)|\leqslant C|x-x_0|\quad \forall x\in B(x_0,
\varepsilon_0)\,,$$
%
другими словами, $f$ -- липшицево в точке $x_0.$ }
\end{corollary}

\medskip
Следующий результат указывает на то, что утверждение теоремы 1.8 в
\cite{GG} является частным случаем теоремы \ref{th4}, по крайней
мере, при $Q\geqslant 1$ и произвольном $n\geqslant 2,$ либо при
$Q\geqslant 0$ при $n=2.$

\medskip
\begin{corollary}\label{cor2}
{\sl\,Пусть $f:{\Bbb B}^n\rightarrow {\Bbb B}^n,$ $f(0)=0$ --
открытое дискретное отображение, удовлетворяющее условию~(\ref{eq1})
в нуле, такое, что
\begin{equation}\label{eq17}
\limsup\limits_{r\rightarrow
0}\int\limits_{r<|x|<1}\frac{Q(x)-1}{|x|^n}\,\,dm(x)<\infty\,.
\end{equation}

Тогда найдётся постоянная $C>0,$ зависящая только от $n$ и $Q,$
такая, что
$$|f(x)|\leqslant C|x|\quad\forall\,\,x\in {\Bbb B}^n\,.$$
В частности, утверждение следствия \ref{cor2} имеет место, если $f$
-- квазиконформное отображение, а $Q(x)$ -- его касательная
дилатация (см. \cite[формула~(1.5) и лемма~2.4]{GG})}.
\end{corollary}

\medskip
\begin{proof}
Из (\ref{eq17}) следует, что для произвольного $0<\varepsilon_0<1$
\begin{equation}\label{eq18}
\limsup\limits_{r\rightarrow
0}\int\limits_{r<|x|<\varepsilon_0}\frac{Q(x)-1}{|x|^n}\,\,dm(x)<\infty\,.
\end{equation}
Достаточно установить, что условие (\ref{eq18}) влечёт неравенство
(\ref{eq15}). По теореме Фубини интеграл в (\ref{eq18}) может быть
записан в виде
$$\int\limits_{r<|x|<\varepsilon_0}\frac{Q(x)-1}{|x|^n}\,\,dm(x)=
\int\limits_{r}^{\varepsilon_0}\int\limits_{S(0,
r)}\frac{Q(x)-1}{|x|^n}dSdr=$$
\begin{equation}\label{eq19}
\omega_{n-1}\cdot\int\limits_{r}^{\varepsilon_0}\frac{q_{0}(r)-1}{r}dr\,.
\end{equation}
Покажем, что почти всюду
\begin{equation}\label{eq16}
\int\limits_t^{\varepsilon_0}\left(1-\frac{1}{q^{1/(n-1)}_0(r)}\right)\cdot
\frac{dr}{r}\leqslant K\cdot
\omega_{n-1}\cdot\int\limits_{t}^{\varepsilon_0}\frac{q_{0}(r)-1}{r}dr\,,
\end{equation}
где $K=\frac{1}{(n-1)\omega_{n-1}}.$ Для этой цели сравним между
собой подинтегральные части выражений в (\ref{eq16}). Прежде всего,
напомним, что при $\lambda\geqslant -1$ имеет место следующее
неравенство Бернулли:
\begin{equation}\label{eq1C}
(1+\lambda)^n\geqslant 1+n\lambda\qquad\forall\,\,n\in {\Bbb N}\,,
\end{equation}
которое непосредственно проверяется методом математической индукции.
Далее, установим неравенство
\begin{equation}\label{eq20}
\left(1-\frac{1}{q^{1/(n-1)}_0(r)}\right)\cdot
\frac{1}{r}\leqslant\frac{1}{n-1}\cdot\frac{q_{0}(r)-1}{r}\,.
\end{equation}
В самом деле, (\ref{eq20}) эквивалентно соотношению
$$
1-\frac{1}{q^{1/(n-1)}_0(r)}\leqslant \frac{1}{n-1}\cdot
(q_{0}(r)-1)\,.
$$
При $q_0(r)=0$ или $q_0(r)=+\infty$ последнее соотношение, а значит,
и (\ref{eq20}), очевидно. Пусть теперь $0<q_0(r)<\infty.$ Обозначив
$s=q_0^{1/(n-1)}(r),$ (\ref{eq20}) можно переписать в виде
$$1-\frac{1}{s}\leqslant \frac{1}{n-1}\cdot(s^{n-1}-1)\,.$$
Домножая последнее соотношение на $s,$ получаем
$$s-1\leqslant \frac{1}{n-1}\cdot(s^n-s)\,,$$
что в свою очередь неравенствам
$$(n-1)(s-1)\leqslant s^n-s$$
и
\begin{equation}\label{eq2C}
s^n\geqslant ns-n+1\,.
\end{equation}
Но (\ref{eq2C}) есть соотношение (\ref{eq1C}), в котором положено
$s=\lambda+1.$ Таким образом, соотношение (\ref{eq2C}) имеет место,
а значит, выполняется эквивалентное ему неравенство (\ref{eq20}).
Однако, из (\ref{eq20}) немедленно следует (\ref{eq16}). Комбинируя
(\ref{eq19}) и (\ref{eq16}), мы получаем, что
$$\int\limits_t^{\varepsilon_0}\left(1-\frac{1}{q^{1/(n-1)}_{x_0}(r)}\right)\cdot
\frac{dr}{r}\leqslant K\cdot
\int\limits_{t<|x|<\varepsilon_0}\frac{Q(x)-1}{|x|^n}\,\,dm(x)\,.$$
Оставшаяся часть утверждения вытекает из последнего соотношения с
учётом предположения в (\ref{eq17}) и следствия \ref{cor1}.~$\Box$
\end{proof}

\medskip
\section{Об отображениях, характеристика которых имеет конечное среднее
значение} В монографии \cite[лемма~6.1]{MRSY} установлены некоторые
полезные свойства функций конечного среднего колебания, которые в
дальнейшем используются при решении различных вопросов теории
отображений (см. напр., \cite{BGR}, \cite{GRY*} и \cite{KRSS}). В
настоящем разделе мы рассмотрим функции с более жёстким
ограничением, когда их среднее значение по малым шарам конечно. Как
будет видно, такие условия существенно <<более выгодны>>, поскольку
их использование позволяет получить качественно более сильные
результаты.

Используя подход, применённый при доказательстве леммы~6.1 в
\cite{MRSY}, мы установим далее некоторое интегральное соотношение
для функций с конечным средним значением по шарам с некоторым
<<весом>>. Затем указанный результат будет применён для решения
проблем локального и граничного поведения одного класса отображений.
Следует отметить, что основные результаты здесь получаются как
следствия из этого свойства и наших более ранних результатов.

\medskip
Всюду далее $h$ -- хордальное расстояние в $\overline{{\Bbb
R}^n}={\Bbb R}^n\cup\{\infty\}$ (см.~\cite[определение~12.1]{Va}).
Имеют место следующие утверждения.

\medskip
\begin{theorem}\label{th1A}
{\sl Пусть $f:D \rightarrow \overline{{\Bbb R}^n}$ --
го\-ме\-о\-мор\-физм, удовлетворяющий условию~(\ref{eq1}) в точке
$x_0\in D,$ такой, что $h\left(\overline{{\Bbb R}^n}\setminus
f\left(D\right)\right)\geqslant\delta>0.$ Предположим, найдётся
постоянная $0<C<\infty$ такая, что
\begin{equation}\label{eq1D}
\limsup\limits_{\varepsilon\rightarrow
0}\frac{1}{\Omega_n\cdot\varepsilon^n}\int\limits_{B(x_0,
\varepsilon)}Q(x)\,dm(x)\leqslant C\,.
\end{equation}
Тогда найдётся $\varepsilon(x_0)\in (0, \infty)$ такое, что для всех
$x\in B(x_0, \varepsilon(x_0))$ имеет место неравенство
\begin{equation}\label{eq2.4.3}
 h(f(x),f(x_0))\leqslant \alpha_n\cdot |x-x_0|^{\beta_n}\,,\end{equation}
где постоянная $\alpha_n$ зависит только от $n$ и $\delta,$ а
$\beta_n$ -- только от $n$ и постоянной $C$ в~(\ref{eq1D}).}
\end{theorem}

\begin{theorem}\label{th2A}
{\sl Пусть $f:D \rightarrow B(0, r),$ $r>0,$ -- открытое дискретное
отображение, удовлетворяющее условию~(\ref{eq1}) в точке $x_0\in D.$
Предположим, найдётся постоянная $0<C<\infty$ такая, что выполнено
соотношение (\ref{eq1D}).
Тогда найдётся $\varepsilon(x_0)\in (0, \infty)$ такое, что для всех
$x\in B(x_0, \varepsilon(x_0))$ имеет место
неравенство~(\ref{eq2.4.3}), где постоянная $\alpha_n$ зависит
только от $n$ и $r,$ а $\beta_n$ -- только от $n$ и постоянной $C$ в
(\ref{eq1D}).}
\end{theorem}

Мы докажем теоремы~\ref{th1A} и \ref{th2A} после того, как будет
доказано ключевое утверждение, сформулированное ниже. Пусть $a>0$ и
$\varphi\colon [a,\infty)\rightarrow[0,\infty)$~--- неубывающая
функция, такая что при некоторых постоянных $\gamma>0,$ $T>0$ и всех
$t\geqslant T$ выполнено неравенство вида
\begin{equation}\label{eq1B}
\varphi(2t)\leqslant \gamma\cdot\varphi(t)\,.
\end{equation}
Несложные примеры таких функций хорошо известны: 1)
$\varphi(t)=t^{\alpha},$ $\alpha\geqslant 0;$ 2) $\varphi(t)=\log
t;$ 3) $\varphi(t)=t^{\alpha}+\log^{\beta}t$ и
$\varphi(t)=t^{\alpha}\log^{\beta}t,$ где $\alpha, \beta\geqslant
0.$ Вообще, можно указать достаточно большое количество функций с
указанным свойством (см.~\cite[разд.~4, \S\,4, гл.~I]{KrRu}).
Условимся называть такие функции {\it функциями, удовлетворяющими
условию удвоения}.

Пусть теперь $\varphi\colon [a,\infty)\rightarrow[0,\infty)$ --
функция с условием удвоения, тогда функция
$\widetilde{\varphi}(t):=\varphi(1/t)$ не возрастает и определена на
полуинтервале $(0, 1/a].$ Имеет место следующее ключевое
утверждение.

\medskip
\begin{lemma}\label{lem1A}
{\sl Пусть 
$a>0,$ $\varphi\colon [a,\infty)\rightarrow[0,\infty)$ --
неубывающая функция с условием удвоения (\ref{eq1B}), $x_0\in {\Bbb
R}^n,$ $n\geqslant 2,$ и пусть $Q:{\Bbb R}^n\rightarrow [0, \infty]$
-- измеримая по Лебегу функция, удовлетворяющая при некотором
$0<C<\infty$ условию
\begin{equation}\label{eq1AA}
\limsup\limits_{\varepsilon\rightarrow
0}\frac{\varphi(1/\varepsilon)}{\Omega_n\cdot\varepsilon^n
}\int\limits_{B(x_0, \varepsilon)}Q(x)\,dm(x)\leqslant C\,.
\end{equation}
%
Тогда найдётся $\varepsilon_0>0$ такое, что
%
$$\int\limits_{\varepsilon<|x-x_0|<\varepsilon_0}
\frac{\varphi(1/|x-x_0|)Q(x)\,dm(x)}{|x-x_0|^n}\leqslant
C_1\cdot\left(\log\frac{1}{\varepsilon}\right)\,,\qquad\varepsilon\rightarrow
0\,,$$
где $C_1:=\frac{\gamma C\Omega_n2^n}{\log 2}.$
}
\end{lemma}

\begin{proof}
Возьмём за основу подход, использованный при доказательстве
\cite[лемма~6.1]{MRSY}. Не ограничивая общности, можно в дальнейшем
положить $x_0=0.$ Так как по условию выполнено
соотношение~(\ref{eq1AA}), то найдётся $\varepsilon_0>0$ такое, что
$Q\in L^1(B(0, \varepsilon_0)).$ Можно считать, что
$1/a<\varepsilon_0,$ где $a$ -- число из области определения функции
$\varphi$ и, в частности, число $1/T$ из (\ref{eq1B}) также меньше
$\varepsilon_0.$

\medskip
Заметим также, что ввиду соотношения~(\ref{eq1AA})
\begin{equation}\label{eq3A}
\delta:=\sup\limits_{r\in (0,
\varepsilon_0)}\frac{\varphi(1/r)}{\Omega_n\cdot
r^n}\int\limits_{B(0, r)}Q(x)\,dm(x)<\infty\,.
\end{equation}
%
Пусть теперь $\varepsilon<2^{\,-1}\varepsilon_0,$
$\varepsilon_k=2^{\,-k+1}\varepsilon_0,$ $A_k=\{ x\in{\Bbb R}^n:
\varepsilon_{k+1}\leqslant |x|<\varepsilon_k\},$ $B_k=B(0,
\varepsilon_k).$ 
Пусть также $N$ -- натуральное число такое, что
$\varepsilon\in[\varepsilon_{N+1},\varepsilon_N).$ Обозначим
$$\alpha(t):=t^{\,-n}\cdot\varphi(1/t),\quad 0<t<\min\{1, a\}\,,$$$$
A(\varepsilon, \varepsilon_0)=\{x\in {\Bbb R}^n:
\varepsilon<|x|<\varepsilon_0\}\,.$$
Заметим, что
$$A(\varepsilon,
\varepsilon_0)\subset A(\varepsilon_{N+1}, \varepsilon_0)
=\bigcup\limits_{k=1}^NA_k\,.$$
%
Поскольку $A_k\subset B_k,$
\begin{equation}\label{eq5A}
|x|^{\,-n}\leqslant \Omega_n2^n/m(B_k)\qquad\forall\,\, x\in A_k\,,
\end{equation}
где, как обычно, $\Omega_n$ обозначает меру единичного шара ${\Bbb
B}^n$ в ${\Bbb R}^n.$ Кроме того, для произвольного $x\in A_k$ с
учётом условия удвоения (\ref{eq1B}), будем иметь: $|x|^n\geqslant
2^{-n}\varepsilon_k^n,$ откуда $|x|\geqslant \varepsilon_k/2$ и
\begin{equation}\label{eq7}
\varphi(1/|x|)\leqslant \varphi(2/\varepsilon_k)\leqslant
\gamma\cdot\varphi(1/\varepsilon_k)\,,\qquad\forall\,\, x\in A_k\,.
\end{equation}
Учитывая соотношения (\ref{eq5A}) и (\ref{eq7}), а также определение
числа $\delta$ в (\ref{eq3A}), мы получим, что
$$\eta(\varepsilon):=\int\limits_{A(\varepsilon_{N+1}, \varepsilon_0)}Q(x)\alpha(|x|)\,dm(x)=
\sum\limits_{k=1}^N\int\limits_{A_k}Q(x)\varphi(1/|x|)|x|^{\,-n}\,dm(x)\leqslant$$
\begin{equation}\label{eq4A}
\leqslant\gamma\Omega_n2^n
\sum\limits_{k=1}^N\frac{\varphi(1/\varepsilon_k)}{m(B_k)}\int\limits_{A_k}Q(x)\,dm(x)\leqslant
\end{equation}
$$\leqslant \gamma\Omega_n2^n
\sum\limits_{k=1}^N\frac{\varphi(1/\varepsilon_k)}{\Omega_n\cdot
\varepsilon_k^n}\int\limits_{B_k}Q(x)\,dm(x)\leqslant
\gamma\Omega_n2^nN\delta\,.$$
Пусть $\varepsilon_0\in (0,2^{\,-1}).$ Поскольку
$\varepsilon<\varepsilon_N$ по выбору $N,$ то
\begin{equation}\label{eq6A}
 N\ <\ N\ +\ \log_2\ \left( \frac{1}{2\varepsilon_0} \right)\ =\
\log_2\ \frac{1}{\varepsilon_N}\ <\ \log_2\
\frac{1}{\varepsilon}=\frac{\log\frac{1}{\varepsilon}}{\log 2}\,.
\end{equation}
Объединяя (\ref{eq4A}) и (\ref{eq6A}), мы получим, что
$$\eta(\varepsilon):=\int\limits_{A(\varepsilon_{N+1}, \varepsilon_0)}Q(x)\alpha(|x|)\,dm(x)\leqslant
\frac{\gamma\delta\Omega_n2^n}{\log
2}\cdot\log\frac{1}{\varepsilon}\,.$$
Обозначив $C_1:=\frac{\gamma C\Omega_n2^n}{\log 2}$ и вспоминая о
том, что $\alpha(t)=t^{-n}\cdot\varphi(1/t),$ последнее соотношение
можно переписать в виде
$$\int\limits_{A(\varepsilon_{N+1}, \varepsilon_0)}\frac{\varphi(1/|x|)Q(x)}{|x|^n}\,dm(x)\leqslant
C_1\cdot\log\frac{1}{\varepsilon}\,,$$
откуда, поскольку $A(\varepsilon, \varepsilon_0)\subset
A(\varepsilon_{N+1}, \varepsilon_0),$ следует, что
$$\int\limits_{\varepsilon<|x|<\varepsilon_0}\frac{\varphi(1/|x|)Q(x)}{|x|^n}\,dm(x)\leqslant
C_1\cdot\log\frac{1}{\varepsilon}\,.$$
Поскольку $\varepsilon\in (0, \varepsilon_0)$ -- произвольно,
утверждение леммы установлено.~$\Box$
\end{proof}

{\it Доказательство} теоремы \ref{th1A} вытекает из леммы
\ref{lem1A} и \cite[лемма~7.6]{MRSY} при $\varphi(t)\equiv 1,$
$\psi(t)=1/t;$ теоремы \ref{th2A} -- из леммы \ref{lem1A} и
\cite[лемма~5]{Sev$_1$} при $\varphi(t)\equiv 1,$
$\psi(t)=1/t.$~$\Box$

\medskip
Основные результаты заметки могут быть применены в области
дифференциальных уравнений, в частности, уравнений Бельтрами на
плоскости (см., напр., \cite{BGR}, \cite{GRY*}).

\section{О граничном поведении решений уравнений Бельтрами}

В недавних работах, посвящённых изучению краевых задач для уравнения
Бельтрами в анизотропных и неоднородных средах, использовалась
логарифмическая ёмкость (см., например, \cite{GRY}--\cite{RE}). Как
известно, логарифмическая ёмкость совпадает с так называемым
трансфинитным диаметром множества. Из этой геометрической
характеристики следует, что множества нулевой ёмкости (в частности,
функции, измеримые относительно логарифмической ёмкости) инвариантны
при непрерывных по Гёльдеру отображениях. Это обстоятельство
является мотивировкой нашего исследования, которому посвящён данный
раздел статьи.

\medskip
В дальнейшем, $D$ --  область комплексной плоскости ${\Bbb C},$
т.е., связное открытое подмножество ${\Bbb C}$. Пусть $\mu\colon
D\rightarrow {\Bbb C}$ -- измеримая функция такая, что $|\mu(z)|<1$
почти всюду в $D.$ {\it Уравнением Бельтрами} называется уравнение
вида
\begin{equation}\label{eq1E}
f_{\bar z}=\mu(z)\,f_z\,,
\end{equation}
где $f_{\bar z}=\overline{\partial}f=(f_x+if_y)/2$, $f_{z}=\partial
f=(f_x-if_y)/2$, $z=x+iy$, $f_x$ и $f_y$ частные производные $f$ по
$x$ и $y$, соответственно. Функция $\mu$ называется {\it комплексной
дилатацией}, а
\begin{equation}\label{eq5}
K_{\mu}(z)=\frac{1+|\mu(z)|}{1-|\mu(z)|}
\end{equation}
-- {\it максимальной дилатацией} уравнения~(\ref{eq1E}). {\it
Якобиан} $J_f(z)$ сохраняющего ориентацию и отображения $f$ в точке
$z,$ имеющего частные производные в этой точке, вычисляется по
правилу $J_f(z)=|f_z|^2-|f_{\overline{z}}|^2.$ В дальнейшем запись
$K_{\mu}(z)$ может использоваться в следующих двух значениях: если
задана функция $\mu,$ то мы вычисляем $K_{\mu}$ посредством
формулы~(\ref{eq5}); если же задано отображение $f,$ имеющее частные
производные, то мы вычисляем $\mu=\mu_f$ по правилу
$$
\mu(z)=\mu_f(z)=\left \{\begin{array}{rr}  |f_z|/|f_{\overline{z}}| , & f_z\ne 0\,,   \\
0, &  f_z=0
\end{array} \right.
$$
при этом, согласно формуле~(\ref{eq5}), полагаем:
$K_{\mu}(z)=K_{\mu_f}(z).$ Уравнение~(\ref{eq1E}) называется {\it
вырожденным}, если ${\rm ess}\,{\rm sup}\,K_{\mu}=\infty.$
Существование гомеоморфных решений класса Соболева  $W^{1,1}_{\rm
loc}$ установлено для вырожденных уравнений Бельтрами при
соответствующих условиях на $K_{\mu}$, см., напр., монографии
\cite{MRSY} и \cite{GRSY}. В дальнейшем используются следующие
стандартные обозначения для кругов и окружностей в комплексной
плоскости:
$${\Bbb D}:=D(0, 1), \quad  D(z_0, r):=\left\{z\in{\Bbb C}:
|z-z_0|< r\right\}\,.$$
Важнейшим инструментом исследования уравнений Бе\-ль\-тра\-ми
является метод модулей. Следующее утверждение связывает решения
этого уравнения с классом кольцевых $Q$-отображений,
см.~\cite[теорема~3.1]{LSS}, \cite[теорема~5.3]{KSS} и
\cite[теорема~1]{KPR}.

\medskip
\begin{proposition}\label{pr1}{\sl Пусть $D$ и  $D^{\,\prime}$ ---
области в ${\Bbb C}$ и $f:D\rightarrow D^{\,\prime}$ ---
гомеоморфное решение уравнения~(\ref{eq1E}) такое, что $f\in
W^{1,1}_{\rm loc}$ и $K_{\mu}\in L^1_{\rm loc}(D)$. Тогда $f$
является кольцевым $Q$-гомеоморфизмом в каждой точке $z_0\in D$ при
$Q(z)=K_{\mu}(z)$.}
\end{proposition}

\medskip
Следующее утверждение относится к вопросу о непрерывном продолжении
решений~(\ref{eq1E}) на единичную окружность.
Согласно~\cite[следствие~7.4]{RSSY}, имеет место следующее

\medskip
\begin{proposition}\label{pr2}
{\sl Пусть $\mu:{\Bbb D} \rightarrow  {\Bbb C}$ --- измеримая в
$\Bbb D$ функция такая, что $|\mu(z)|<1$ п.в., и пусть также
$f:{\Bbb D}\rightarrow{\Bbb D}$  --- гомеоморфное решение
уравнения~(\ref{eq1E}), принадлежащее $W_{\rm loc}^{1,1}$. Если для
всех $\zeta\in\partial \Bbb D$ выполнено условие
%
$$\limsup\limits_{\varepsilon\rightarrow
0}\frac{1}{\pi\varepsilon^2}\int\limits_{{\Bbb D}\cap D(\zeta,
\varepsilon) }K_{\mu}(z)\, dm(z)<\infty\,,
$$
то $f$ имеет гомеоморфное продолжение  $f:\overline{\Bbb D}\rightarrow\overline{\Bbb D}.$}
\end{proposition}

\medskip
Нам также понадобится результат о локальном поведении кольцевых
$Q$-гомеоморфизмов во внутренних точках области,
см.~\cite[следствие~4.1]{LSS}.

\medskip
\begin{proposition}\label{pr3} {\sl Пусть $D$  и $D^{\,\prime}$  ---
области в $\Bbb C$,  $Q: D\rightarrow [0,\infty]$ --- измеримая по
Лебегу функция, и пусть $f:D\rightarrow D^{\,\prime}$ --- кольцевой
$Q$-гомеоморфизм в точке $z_0\in D$ такой, что $h(\overline{{\Bbb
C}}\setminus f(D))\geqslant \Delta>0.$ Предположим, что существуют
$0<\varepsilon_0< {\rm dist}(z_0, \partial D)$ и измеримая по Лебегу
функция $\psi:(0,\infty)\rightarrow [0,\infty],$ удовлетворяющая
условию
$$0<I(\varepsilon):=\int\limits_{\varepsilon}^{\varepsilon_0}\,
\psi(t)\,dt <\infty \qquad \forall\,\,\varepsilon\in(0,
\varepsilon_0)\,,$$
при этом,
\begin{equation}\label{eq2}
\int\limits_{A(z_0, \varepsilon, \varepsilon_0)}
Q(z)\,\psi^2(|z-z_0|)\,dm(z)\,\leqslant C\cdot I(\varepsilon)\qquad
\forall\,\,\varepsilon\in(0, \varepsilon_0)\,.
\end{equation}
Тогда для всех $z\in D(z_0, \varepsilon_0)$
%
$$h( f(z), f(z_0))\leqslant \frac{32}{\Delta}\exp \left\{-
\frac{2\pi}{C}I(|z-z_0|)\right\}\,.$$
} \end{proposition}

\medskip
Полагая $\psi(t)=\frac{1}{t}$ в формулировке предложения~\ref{pr3},
получаем следующее

\medskip
\begin{corollary}\label{cor3}
{\sl Предположим, что в условиях предложения~\ref{pr3} вместо
соотношения~(\ref{eq2}) имеет место условие
\begin{equation}\label{eq3}
\int\limits_{A(z_0, \varepsilon, \varepsilon_0)}
\frac{Q(z)}{|z-z_0|^2}\, dm(z)\,\leqslant C
\log\left(\frac{\varepsilon_0}{\varepsilon}\right)\qquad
\forall\,\,\varepsilon\in(0, \varepsilon_0)\,. \end{equation}
Тогда для всех $z\in D(z_0, \varepsilon_0)$
%
$$h( f(z), f(z_0))\leqslant \frac{32}{\Delta}\varepsilon_0^{-2\pi/C}
|z-z_0|^{2\pi/C}\,.$$
 }
\end{corollary}

\medskip
Следующая лемма связывает среднее значение функции по кругу и
интеграл, участвующий в левой части неравенства~(\ref{eq3}).

\medskip
\begin{lemma}\label{lem2}
{\sl Пусть  $Q:{\Bbb C}\rightarrow [0, \infty]$
--- измеримая по Лебегу функция, при этом, существуют $0<\delta_0<1$ и
постоянная $C_*>0$ такие, что
%
$$\sup\limits_{r\in(0,\delta_0)} \frac{1}{\pi
r^2}\int\limits_{B(z_0,r)} Q(z)\,\,dm(z) < C_*\qquad \forall\,\,
z_0\in \partial {\Bbb D}\,.$$
%
Тогда найдётся $\varepsilon_0>0$ такое, что для всех $\varepsilon\in
(0, \varepsilon_0)$ и $z_0\in
\partial {\Bbb D}$
%
$$\int\limits_{A(z_0, \varepsilon, \varepsilon_0)}
\frac{Q(z)\,dm(z)}{|z-z_0|^2} \leqslant \frac{4\pi C_*}{\log
2}\log\frac{1}{\varepsilon}\,.$$
}
\end{lemma}

\medskip
{\it Доказательство леммы~\ref{lem2}} немедленно следует из
леммы~\ref{lem1A} при $n=2,$ $\varphi(t)\equiv 1$ и $C_*=C.$~$\Box$

\medskip
Ключевым утверждением данного раздела является следующая

\medskip
\begin{lemma}\label{lem3} {\sl Пусть $\mu:{\Bbb D}\rightarrow {\Bbb C}$ ---
измеримая по Лебегу функция, удовлетворяющая условию $|\mu(z)|<1$
для п.в. $z\in \Bbb D.$ Предположим, $f_0:{\Bbb D}\rightarrow{\Bbb
D},$ $f_0({\Bbb D})={\Bbb D},$ -- гомеоморфное решение уравнения
Бельтрами~(\ref{eq1E}), принадлежащее пространству $W_{\rm
loc}^{1,1}.$ 
Если $K_{\mu}\in L^1(\Bbb D)$ и, кроме того, найдутся
$\varepsilon_0\in(0,1)$ и $C\in[1,\infty),$ такие что
\begin{equation}\label{eq1AE}
\sup\limits_{\varepsilon\in(0,\varepsilon_0)}\frac{1}{\pi\varepsilon^2}\int\limits_{{\Bbb
D}\cap D(\zeta,\varepsilon)}K_{\mu}(z)\ dm(z) < C\qquad \forall\,\,
\zeta \in
\partial {\Bbb D}\,,
\end{equation}
то $f_0$ имеет гомеоморфное продолжение $f:\overline{\Bbb
D}\rightarrow \overline{\Bbb D},$ при этом,
$$
|f(z_2)-f(z_1)| \leqslant 64 \, \varepsilon_0^{\,-\alpha}
|z_2-z_1|^\alpha \qquad \forall\,\, z_1, z_2\in\partial{\Bbb D}:
|z_2-z_1|<\delta_0\,,
$$
где $\delta_0:=\min\left\{\frac{1}{2},\varepsilon_0^2\right\}$ и $
\alpha:=\frac{\log 2}{68 C}.$ }
\end{lemma}

\medskip
\begin{proof} Заметим, что $f_0$ допускает
гомеоморфное продолжение $f:\overline{\Bbb D}\rightarrow
\overline{\Bbb D}$ в силу предложения~\ref{pr2}. Продолжим $f$ по
симметрии на внешность круга $\Bbb D,$ полагая
$$
F(z)=  \left \{\begin{array}{rr}  f(z) , & |z|<1\,,   \\
1/\overline{f\left(1/\overline{z}\right)}, &  \  |z|>1\,.
\end{array} \right.
$$
При помощи прямых вычислений несложно установить, что комплексная
дилатация отображения $F$ имеет вид
$$
\mu_F(z)=  \left \{\begin{array}{rr}  \mu(z) , & |z|<1\,,   \\
\frac{z^2}{\overline{z}^2}
\overline{\mu\left(1/\overline{z}\right)}, &  \  |z|>1\,.
\end{array} \right.
$$

\medskip
По условию, $f_0\in W_{\rm loc}^{1,1}(\Bbb D),$ поэтому также $F\in
W_{\rm loc}^{1,1}(\Bbb D).$ Покажем большее, а именно, что $F\in
W^{1,1} (\Bbb D)$ (другими словами, мы утверждаем, что частные
производные отображения $F$ не только локально, но и глобально
интегрируемы в $(\Bbb D)$). Для этого заметим, что
$$
|F_{\overline{z}} | \leqslant | F_z | \leqslant | F_z | +|
F_{\overline{z}} |\, \leqslant K^{\frac{1}{2}}_{\mu_F}(z)\, \,
J^{\frac{1}{2}}_{F}(z)\,.
$$
Отсюда
$$
\int\limits_{\Bbb D} \, | F_z |  \, dm(z)\, \leqslant \,
\int\limits_{\Bbb D} \, K^{\frac{1}{2}}_{\mu_F}(z)\, \,
J^{\frac{1}{2}}_{F}(z)\, dm(z) \,.
$$
Далее, применяя неравенство Гельдера, получаем
\begin{equation}\label{eq28}
\int\limits_{\Bbb D} \,|F_z|\,dm(z)\leqslant
 \left(\int\limits_{\Bbb D} \,
 K_{\mu}(z)\,dm(z)\right)^{\frac{1}{2}}\cdot
\left(\int\limits_{\Bbb D}\,J_F(z)\,dm(z)\right)^{\frac{1}{2}} \,.
\end{equation}
В силу гомеоморфности отображения  $F$, имеем
\begin{equation}\label{eq29}
\int\limits_{\Bbb D} \,J_F(z)\,dm(z)\leqslant m(F({\Bbb D}))=\pi\,
\end{equation}
см.~\cite[теоремы~3.1.4, 3.1.8 и 3.2.5]{Fe}. Учитывая условие
$K_{\mu}\in L^1(\Bbb D)$, из оценок~(\ref{eq28}) и~(\ref{eq29})
получаем, что
\begin{equation}\label{eq310}
\int\limits_{\Bbb D} \, |F_z|\,dm(z) \leqslant \left(\pi
\int\limits_{\Bbb D} \,
K_{\mu}(z)\,dm(z)\right)^{\frac{1}{2}}<\infty\,.
\end{equation}
Покажем, теперь, что $F\in W^{1,1} (D_R) $ для любого $R>1$, где
$D_R:=D(0,R)$. Действительно, по аддитивности интеграла Лебега,
имеем равенство
%
$$\int\limits_{|z|\leqslant R}\,|F_z|\,dm(z)= \int\limits_{\Bbb
D}|F_z|\,dm(z)+\int\limits_{1\leqslant|z|\leqslant
R}\,|F_z|\,dm(z)\,.$$
%
%
В силу (\ref{eq310}) имеем: $\int\limits_{\Bbb D}\,|F_z|\,
dm(z)<\infty$.

Заметим, что $F\in ACL({\Bbb C}).$ В самом деле, единичный круг
можно разбить на не более, чем счётное число прямоугольников, в
которых $F$ абсолютно непрерывна на почти всех на координатных
отрезках. Применив в каждом из прямоугольников критерий абсолютной
непрерывности через интегрирование производных,
см.~\cite[теорема~IV.7.4]{Sa}, и воспользовавшись непрерывностью $f$
в $\overline{\Bbb D},$ мы получим абсолютную непрерывность $f$ на
отрезках в $\overline{\Bbb D}$, одна из концевых точек которых может
принадлежать единичной окружности. Аналогично можно рассуждать для
области ${\Bbb C}\setminus D.$ В таком случае, абсолютная
непрерывность на линиях $F$ в ${\Bbb C}$ вытекает из
критерия~\cite[теорема~IV.7.4]{Sa} и аддитивности одномерного
интеграла Лебега.

\medskip
Осталось показать, что $\int\limits_{1<|z|\leqslant R}
|F_z|\,dm(z)<\infty.$ Заметим, что в силу гомеоморфности отображения
$F$
\begin{equation}\label{eq32}\int\limits_{1
\leqslant |z|\leqslant R} \, J_F(z)\,dm(z)\leqslant
\int\limits_{D_R} \, J_F(z)\,dm(z)\leqslant m(F(D_R))<\infty\,,
\end{equation}
см.~\cite[теоремы~3.1.4, 3.1.8 и 3.2.5]{Fe}. Далее, покажем, что
$$\int\limits_{1<|z|< R} \, K_{\mu_F}(z)\,dm(z)<\infty\,.$$
Сделав замену переменных $w= \frac{1}{\overline{z}}$ и
воспользовавшись~\cite[теорема~3.25]{Fe}, преобразуем этот интеграл
к виду:
$$\int\limits_{1\leqslant|z|\leqslant R} \,
K_{\mu_F}(z)\,dm(z)=\int\limits_{1\leqslant|z|\leqslant R} \,
K_{\mu}\left( \frac{1}{\overline{z}}\right)\,dm(z)=$$
$$
=\int\limits_{1/R\leqslant|z|\leqslant 1} \,
K_{\mu}(w)\,\frac{dm(w)}{|w|^4}\,.
$$
Следовательно,
\begin{equation}\label{eq34}
\int\limits_{1\leqslant|z|\leqslant R}\,K_{\mu_F}(z)\,dm(z)\leqslant
R^4 \int\limits_{{\Bbb D}}\, K_{\mu}(w)\, dm(w)<\infty\,.
\end{equation}
Применяя неравенство Гёльдера и оценки~(\ref{eq32}), (\ref{eq34}),
получаем:
$$
\int\limits_{1\leqslant|z|\leqslant R}|F_z|\,dm(z)\leqslant
\int\limits_{1\leqslant |z|\leqslant R} K^{\frac{1}{2}}_{\mu_F}(z)\,
\,J^{\frac{1}{2}}_{F}(z)\,dm(z) \leqslant
$$
$$\leqslant \,\left(  \int\limits_{1 \leqslant |z|\leqslant R} \,
K_{\mu_F}(z)\,\, dm(z)\right) ^{\frac{1}{2}}\left(\int\limits_{1
\leqslant |z|\leqslant R} \,  J_{F}(z)\, dm(z)\right)
^{\frac{1}{2}}<\,\infty\,.
$$
Включение $F\in W_{\rm loc}^{1,1}(\Bbb C)$ установлено.

\medskip
Оценим теперь интеграл $\int\limits_{ D(\zeta, \varepsilon)
}K_{\mu_F}(z)\, dm(z)$ при $\varepsilon\in (0, \varepsilon_0)$. Для
этого, разобьём его на две части:
$$\int\limits_{ D(\zeta, \varepsilon) }K_{\mu_F}(z)\, dm(z)=$$
\begin{equation}\label{eq38}
=\int\limits_{{\Bbb D}\cap D(\zeta, \varepsilon) }K_{\mu_F}(z)\,
dm(z)+\int\limits_{D(\zeta, \varepsilon)\setminus {\Bbb D}
}K_{\mu_F}(z)\, dm(z)\,.
\end{equation}
Сделав замену переменных $w= \frac{1}{\overline{z}}$ и
воспользовавшись~\cite[теорема~3.25]{Fe}, преобразуем второй
интеграл к виду:
$$\int\limits_{D(\zeta, \varepsilon)\setminus {\Bbb D} }K_{\mu_F}(z)\,
dm(z)=$$
%
$$=\int\limits_{D(\zeta, \varepsilon)\setminus {\Bbb D} }K_{\mu}
\left(\frac{1}{\overline{z}}\right)\, dm(z)= \int\limits_{ D(\zeta,
\varepsilon)\cap {\Bbb D}} K_{\mu} \left(w\right)\,
\frac{dm(w)}{|w|^4}\,.$$
%
Проверим следующее неравенство:
$$
\max\limits_{w\in D(\zeta, \varepsilon)\cap {\Bbb D}}
\frac{1}{|w|^4}< 16\,,\qquad \varepsilon\in \left(0,\frac{1}{2}
\right)\,.
$$
Действительно,
$$
\max\limits_{|w-\zeta|= \varepsilon} \frac{1}{|w|^2}=
\max\limits_{\varphi\in [0,2\pi)} \frac{1}{|\zeta+\varepsilon
e^{i\varphi}|^2}=\max\limits_{\varphi\in [0,2\pi)}
\frac{1}{|\zeta|^2+2\varepsilon Re(\zeta e^{-i\varphi})
+\varepsilon^2 }=
$$
\begin{equation}\label{eq6}
=\max\limits_{\varphi\in [0,2\pi)} \frac{1}{1+2\varepsilon
\cos(\varphi-\vartheta) +\varepsilon^2 }=
\frac{1}{(1-\varepsilon)^2}<4\,,
\end{equation}
где $w= \zeta+\varepsilon e^{i\varphi}, \zeta=e^{i\vartheta}$.

\medskip
Таким образом, получаем:
$$
\int\limits_{D(\zeta, \varepsilon)\setminus {\Bbb D} }K_{\mu_F}(z)\,
dm(z)\leqslant \max\limits_{w\in D(\zeta, \varepsilon)\cap {\Bbb D}}
\frac{1}{|w|^4}\int\limits_{D(\zeta, \varepsilon)\cap {\Bbb D}
}K_{\mu} \left(w\right)\, dm(w)<$$$$<16 \int\limits_{D(\zeta,
\varepsilon)\cap {\Bbb D} }K_{\mu} \left(w\right)\, dm(w)\,.
$$
Учитывая оценку~(\ref{eq6}) и равенство (\ref{eq38}), мы получим
отсюда, что
\begin{equation}\label{eq7A}
\int\limits_{D(\zeta, \varepsilon) }K_{\mu_F}(z)\, dm(z) <17
\,\int\limits_{D(\zeta, \varepsilon)\cap {\Bbb D}}K_{\mu}(w)\,
dm(w)\,. \end{equation}
Из условий~(\ref{eq1AE}) и (\ref{eq7A}) вытекает, что
$
\sup\limits_{\varepsilon\in(0,\varepsilon_0)}\frac{1}{\pi\varepsilon^2}\int\limits_{
D(\zeta, \varepsilon)} K_{\mu_F}(z)\, dm(z)<17
\sup\limits_{\varepsilon\in(0,\varepsilon_0)}\frac{1}{\pi\varepsilon^2}\int\limits_{
{\Bbb D}\cap D(\zeta, \varepsilon) }K_{\mu}(z)\, dm(z)<17 C.$ Далее,
применяя лемму~\ref{lem2} при $C_*=17 C,$ при некотором
$0<\varepsilon_0<1$ получаем оценку
$$\int\limits_{ A(z_0, \varepsilon, \varepsilon_0)}
\frac{K_{F_{\mu}}(z)\,dm(z)}{|z-z_0|^2} \leqslant \frac{68\pi
C}{\log 2}\left(\log\frac{1}{\varepsilon}\right),\quad\forall\,\,
\varepsilon\in (0, \varepsilon_0)\,,\quad \forall\,\,z_0\in
\partial {\Bbb D}\,.$$

Заметим, что
$ \frac{\log\frac{1}{\varepsilon}} {\log
\left(\frac{\varepsilon_0}{\varepsilon}\right)}= 1+\frac{\log
\frac{1}{\varepsilon_0}} {\log
\left(\frac{\varepsilon_0}{\varepsilon}\right)}<2 $ для всех
$\varepsilon\in (0, \delta_0).$ Тогда
$$\log
\left(\frac{\varepsilon_0}{\varepsilon}\right)^{\,-1}\cdot\int\limits_{
A(z_0, \varepsilon, \varepsilon_0  )}
\frac{K_{F}(z)\,dm(z)}{|z-z_0|^2}\leqslant \frac{68\pi C}{\log 2}\,
\frac{\log\frac{1}{\varepsilon}} {\log
\left(\frac{\varepsilon_0}{\varepsilon}\right)}\leqslant
\frac{136\pi C}{\log 2}\,. $$

Окончательно, полагая $\varepsilon:=|z_2-z_1|$ и применяя
следствие~\ref{cor3} с учётом предложения~\ref{pr1}, получаем
оценку:
%
$$h(f(z_1),f(z_2))\leqslant 32 \,\varepsilon_{0}^{\,-\alpha}
\,|z_1-z_2|^{\alpha}\,,\qquad \alpha=\frac{\log 2}{68 C}\,.$$
%
Поскольку $z_1$ и $z_2\in
\partial {\Bbb D},$ мы получим, что
$$|f(z_1)-f(z_2)|\leqslant 64 \,\varepsilon_{0}^{\,-\alpha}
\,|z_1-z_2|^{\alpha}\,.$$ Лемма доказана.
\end{proof}

\medskip
В завершение нашей статьи рассмотрим следующий результат.

\medskip
\begin{theorem}\label{th3}
{\sl Пусть $\mu:{\Bbb D} \rightarrow  {\Bbb D}$ -- измеримая по
Лебегу функция и $f:{\Bbb D}\rightarrow {\Bbb D}$
--- гомеоморфное решение уравнения~(\ref{eq1E}), принадлежащее классу
$W_{\rm loc}^{1,1}$. Если $K_{\mu}\in L^1(\Bbb D)$ и, кроме того,
найдутся $\varepsilon_0\in(0,1)$ и $C\in[1,\infty)$ такие, что
%
$$\sup\limits_{\varepsilon\in(0,\varepsilon_0)}\frac{1}{\pi\varepsilon^2}\int\limits_{{\Bbb
D}\cap D(\zeta,\varepsilon)}K_{\mu}(z)\, dm(z)< C\qquad \forall\,\,
\zeta \in
\partial {\Bbb D}\,,$$
%
то $f$ имеет гомеоморфное продолжение на $\partial {\Bbb D},$
являющееся непрерывным по Гёльдеру.}
\end{theorem}

\medskip
\begin{proof}
%
В обозначениях леммы~\ref{lem3},  при $|z_1-z_2|\geqslant\delta_0$
имеем тривиальную оценку
%
$|f(z_2)-f(z_1)| \leqslant 2=\frac{2}{\delta_0^{\alpha}}\,
\delta_0^{\alpha}\leqslant
\frac{2}{\delta_0^{\alpha}}\,|z_1-z_2|^{\alpha}.$
Положим $L:=\max\left\{\frac{2}{\delta_0^{\alpha}}, 64
\,\varepsilon_{0}^{-\alpha}\right\}.$ Тогда по лемме~\ref{lem3}
$$|f(z_2)-f(z_1)| \leqslant L\,|z_1-z_2|^{\alpha}\qquad \forall z_1, z_2 \in \partial {\Bbb D}: |z_2-z_1|<\delta_0\,,$$
что и требовалось установить.
\end{proof}

\medskip\medskip\medskip
КОНТАКТНАЯ ИНФОРМАЦИЯ

\medskip
\noindent{\bf Рязанов Владимир Ильич} \\
Институт прикладной математики и механики НАН Украины,
ул.~Добровольского~1, г.~Славянск, Украина, 84100,
\\ e-mail: vlryazanov1@rambler.ru

\medskip
\noindent{\bf Салимов Руслан Радикович} \\
Институт математики НАН Украины, ул. Терещенковская~3, г.~Киев-4,
Украина, 01601, \\ e-mail: ruslan.salimov1@gmail.com

\medskip
\noindent{\bf Севостьянов Евгений Александрович} \\
Житомирский государственный университет имени Ивана Франко, ул.
Большая Бердичевская~40, г.~Житомир, Украина, 10008, \\ e-mail:
esevostyanov2009@gmail.com

\end{document}